\documentclass[letterpaper, 10 pt, conference]{ieeeconf}  

\IEEEoverridecommandlockouts                              
\usepackage{graphicx} 
\usepackage{amsmath} 
\usepackage{amssymb} 

\usepackage{newtxmath}
\usepackage{newtxtext}
\usepackage{pgf}
\usepackage{tikz}
\usepackage{listings}
\usepackage{array}
\usepackage{arydshln}
\usepackage{booktabs}
\usepackage{hyperref}
\usepackage{bookmark}
\newcolumntype{N}{@{}m{0pt}@{}} 
\usetikzlibrary{arrows,positioning,calc}
\def\lstbasicfont{\fontfamily{pcr}\selectfont\footnotesize}
\lstdefinelanguage{matlab}{
    comment=[l]\%,
    morecomment=[l]...,
    morecomment=[s]{\%\{}{\%\}},
}[comments]
\newcommand{\textbrut}{\lstinline[basicstyle={\ttfamily},breaklines=true]} 
\colorlet{mylinkcolor}{red!80!black}
\colorlet{myurlcolor}{green!50!black}
\colorlet{mysectioncolor}{blue!50!black}
\hypersetup{
    colorlinks=true,
    linkcolor=mylinkcolor,
    urlcolor=myurlcolor,
    hypertexnames,
}

\title{\LARGE\bf Path-complete $p$-dominant switching linear systems}

\author{Guillaume~O.~Berger, Fulvio~Forni and Rapha\"{e}l~M.~Jungers%
\thanks{G.~Berger is a FNRS/FRIA Fellow.
\textbrut{guillaume.berger@uclouvain.be}.
R.~Jungers is a FNRS Research Associate.
He is supported by the French Community of Belgium, the Walloon Region and the Innoviris Foundation.
\textbrut{raphael.jungers@uclouvain.be}.
Both are with ICTEAM institute, UCLouvain, Louvain-la-Neuve, Belgium.
F.~Forni is with the Department of Engineering, University of Cambridge, Trumpington Street, Cambridge CB2 1PZ, UK.
\textbrut{f.forni@eng.cam.ac.uk}.}}

\newtheorem{theorem}{Theorem}

\newtheorem{proposition}[theorem]{Proposition}
\newtheorem{remark}{Remark}
\newtheorem{example}{Example}
\newtheorem{definition}{Definition}
\newtheorem{problem}{Problem}

\newcommand{\RRb}{\mathbb{R}}
\newcommand{\ntn}{{n\times n}}
\newcommand{\In}{\mathrm{In}}
\newcommand{\saufzero}{\setminus\{0\}}
\newcommand{\nint}{\mathrm{int}}
\newcommand{\calK}{\mathcal{K}}

\newcommand{\ZZb}{\mathbb{Z}}
\newcommand{\calH}{\mathcal{H}}
\newcommand{\calV}{\mathcal{V}}
\newcommand{\calS}{\mathcal{S}}
\newcommand{\calU}{\mathcal{U}}
\newcommand{\dime}{\mathrm{dim}}

\newcommand{\MyTransition}[3]{\ensuremath{#1\,\raisebox{-2pt}[0pt][0pt]{$\stackrel{#2}{\longrightarrow}$}\,#3}}
\newcommand{\MyTransitionNum}[3]{\ensuremath{#1\,\raisebox{-2pt}[0pt][0pt]{$\stackrel{\raisebox{-2pt}[0pt][0pt]{\scriptsize$#2$}}{\longrightarrow}$}\,#3}}
\newcommand{\ttAut}{\mathtt{Aut}}

\newcommand{\calL}{\mathcal{L}}
\newcommand{\tta}{\mathtt{a}}
\newcommand{\ttb}{\mathtt{b}}
\newcommand{\ttc}{\mathtt{c}}
\newcommand{\ttd}{\mathtt{d}}
\newcommand{\tte}{\mathtt{e}}
\newcommand{\setMatrix}{\small\arraycolsep=0.5\arraycolsep}

\newcommand{\PrintSwicth}{x(t+1) = A_{\sigma(t)}\, x(t)}
\newcommand{\Hhat}{\bar{H}}
\newcommand{\Vhat}{\bar{V}}

\newcommand{\calKP}[1]{\calK(P_{#1})}

\renewcommand{\baselinestretch}{1.03}

\begin{document}

\maketitle
\thispagestyle{empty}
\pagestyle{empty}

\begin{abstract}
The notion of path-complete p-dominance for switching linear systems (in short, path-dominance) is introduced as a way to generalize the notion of dominant/slow modes for LTI systems.
Path-dominance is characterized by the contraction property of a set of quadratic cones in the state space.
We show that path-dominant systems have a low-dimensional dominant behavior, and hence allow for a simplified analysis of their dynamics.
An algorithm for deciding the path-dominance of a given system is presented.
\end{abstract}

\section{INTRODUCTION}
\label{sec-introduction}

Lyapunov methods are ubiquitous in system analysis.
The decay of a positive Lyapunov function along system trajectories guarantees their asymptotic convergence to the fixed point at the minimum of the Lyapunov function.
By Lyapunov methods, linear stability analysis reduces to the feasibility of a few linear matrix inequalities (LMIs).
The nonlinear case is conceptually similar but requires more general classes of Lyapunov functions.
In this paper, we mimic Lyapunov analysis but replace the contracting ellipsoids of quadratic Lyapunov theory with contracting cones.
The aim is to develop tractable methods for the analysis of multistable and oscillatory systems, typically captured by low-order reduced  dynamics.

In our study, we consider switching linear systems whose switches are regulated by a language.
Restricting switches to a language increases the expressiveness of these systems.
For example, the language can be used to model specific sequences of actions/communication/disturbances in distributed computation algorithms, cooperative dynamics, or in systems with uncertainties.
Switching systems constrained to a language are hard to analyze and control due to the complexity of their dynamics.
But many complex systems have in fact a low-dimensional \emph{dominant} behavior.
The aim of this paper is to provide a computational framework for deciding whether a discrete-time switching linear system is $p$-dominant, i.e., has a $p$-dimensional dominant behavior.

For linear time-invariant (LTI) systems, $p$-dominance is equivalent to the existence of $p$ dominant/slow modes and $n-p$ transient/fast modes of the system.
The inequality $A^\top PA - \gamma^2\, P \prec 0$ captures this feature, where $A\in\mathbb{R}^{n\times n}$ is the system transition matrix, $\gamma$ is the dominance \emph{rate} separating dominant and non-dominant modes, and $P$ is a symmetric matrix with $p$ negative eigenvalues and $n-p$ positive eigenvalues.
Geometrically, using indefinite matrices is a way to replace ellipsoids with quadratic cones, i.e., cones that can be described as the set of points $x$ such that $x^\top Px\leq0$.
The feasibility of the Lyapunov inequality thus reads as the contraction of a quadratic cone.
In fact, the seminal example of $p$-dominance for LTI systems is positivity \cite{bib-Lue79IDST}, \cite{bib-Bus73HMPC}.
By Perron-Frobenius theorem \cite{bib-Bus73HMPC}, \cite{bib-Van68SPMH}, positivity of an LTI system implies convergence of the trajectories to a $1$-dimensional subspace.
$p$-dominant linear systems extend the property to $p$-dimensional attractors \cite{bib-ForSep17DTPD}.
Similar results hold for $p$-dominant nonlinear systems \cite{bib-ForSep17DDTD}.

In this paper, we study dominant switching systems.
As a first step, dominance is extended to \emph{path-dominance}, adapting the property to switching systems constrained by a language.
We then combine ideas from path-complete Lyapunov theory \cite{bib-AngAth17PCGC}, \cite{bib-JunAhm17CLIS}, path-complete positivity \cite{bib-ForJun17PCPS}, and dominance \cite{bib-ForSep17DDTD}, to derive a new set of linear matrix inequalities for path-dominance.
The goal is to increase expressiveness by moving from a uniform cone to a family of cones while preserving the feature of a low-dimensional asymptotic behavior.
The approach requires the introduction of an automaton whose transitions are paired to the admissible switches of the system.
Each state of the automaton $q$ is then paired to a quadratic form $P_q$ of fixed inertia, and a set of LMIs is built from the transitions of the automaton.
The result is a sound algorithm that computes a contracting family of cones.

The paper is organized as follows: Section~\ref{sec-path-dominant-systems} discusses the notion of $p$-dominance for LTI systems and extends the property to switching linear systems.
Path $p$-dominance is defined and discussed in detail, and illustrated by an example.
In Section~\ref{sec-behavior}, we describe the asymptotic behavior of $p$-dominant systems.
The algorithm for path-dominance is presented in Section~\ref{sec-algorithm}.
Conclusions follow.

\section{PATH-DOMINANT SYSTEMS}
\label{sec-path-dominant-systems}

\subsection{Dominant LTI systems}

A linear time-invariant (LTI) system
\begin{equation}
\label{eq-LTI-sys}
x(t+1) = A x(t) ,\quad x \in \RRb^n ,~ A\in\RRb^\ntn
\end{equation}
is \emph{$p$-dominant} if it has $p$ dominant/slow modes and $n-p$ transient/fast modes.
The dominant modes characterize the asymptotic behavior of the system, typically captured by dynamics of reduced order \cite{bib-ForSep17DDTD}, \cite{bib-ForSep17DTPD}.
To characterize this property, we use symmetric matrices $P\in\RRb^\ntn$ of fixed \emph{inertia} $\In(P)$, given by the triplet $(i_-,i_0,i_+)$ where $i_-$, $i_0$, and $i_+$ are the number of negative, zero, and positive eigenvalues of $P$, respectively.
Given any symmetric matrix $P$ with inertia $(p,0,n-p)$, we will make use of the associated \emph{quadratic $p$-cone}:
\begin{equation}
\label{eq-quad-cone}
\calK(P)=\{x\in\RRb^n\,:\, x^\top Px\leq0\} .
\end{equation}
The degree $p$ stands for the maximal dimension of a linear subspace contained in $\calK(P)$, e.g., the eigenspace associated to the $p$ negative eigenvalues of $P$.

\begin{definition}\label{def-LTI-dominant}
A linear system \eqref{eq-LTI-sys} is \emph{$p$-dominant} if there exists a quadratic $p$-cone $\calK(P)$ such that
\begin{equation}
\label{eq-contra-single}
A\big(\calK(P)\saufzero\big) \subseteq \nint\, \calK(P)
\end{equation}
where $\nint\, \calK(P)$ denotes the interior of $\calK(P)$.\hfill$\lrcorner$
\end{definition}

When \eqref{eq-contra-single} holds, we say that $\calK(P)$ is \emph{contracted} by $A$.
This geometric property has an algebraic interpretation, based on linear matrix inequalities (LMIs):

\begin{proposition}\label{pro-LMI-S-proc}
Let $A\in\RRb^\ntn$ and consider a $p$-cone $\calK(P)$.
$\calK(P)$ is contracted by $A$ if and only if there exists a \emph{rate} $\gamma>0$ and an $\varepsilon>0$ such that
\begin{equation}
\label{eq-LMI-single}
A^\top PA - \gamma^2\, P \preceq -\varepsilon I.
\end{equation}
\vskip-10pt
\hfill$\lrcorner$
\end{proposition}

\begin{proof}
[$\eqref{eq-LMI-single}\Rightarrow\eqref{eq-contra-single}$]
Let $x\in\calK(P)\saufzero$.
Then $x^\top Px\leq0$.
If $y=Ax$, we have from \eqref{eq-LMI-single} that $y^\top Py\leq \gamma^2\,x^\top Px-\varepsilon\lvert x\rvert^2<0$.
Hence, $y\in\nint\,\calK(P)$.
[$\eqref{eq-LMI-single}\Leftarrow\eqref{eq-contra-single}$]
\eqref{eq-contra-single} states that, for every $x\in\RRb^n$, $x^\top Px\leq0$, $x\neq0$, implies $x^\top\!A^\top P\,Ax<0$.
Therefrom, we deduce \eqref{eq-LMI-single} by applying the $\calS$-Lemma (see, e.g., \cite[Theorem~4.3.3]{bib-BenNem01LMCO}, \cite[Section~B.2]{bib-BoyVan04CO}).
\end{proof}

$1$-dominance is closely related to the property of positivity \cite{bib-Bus73HMPC}, since it implies the existence of a closed solid convex pointed cone contracted by $A$ or $-A$.
Positive systems have been intensively studied in the past decades.
Their transition matrix $A$ has a single dominant eigenvalue \cite{bib-Van68SPMH} and the corresponding eigenvector generates a 1-dimensional attractive invariant subspace of the system.
This fundamental property has been used in a large number of contexts, for example, for the analysis of Markov chains and compartmental systems \cite{bib-Lue79IDST}, \cite{bib-FarRin11PLST}, or for observer design \cite{bib-BonAst11CODC}, \cite{bib-BacAst08DPLO}.
From the equivalence between \eqref{eq-contra-single} and \eqref{eq-LMI-single}, $1$-dominant LTI systems enjoy a similar property.
Quadratic cones allow for a tractable condition \eqref{eq-LMI-single}, and open the way to generalize the approach of positivity to families of cones that are compatible with $p$-dimensional attractors.
For instance, $p$-dominant linear systems \eqref{eq-LTI-sys} have $p$ dominant eigenvalues
\begin{equation}
\label{eq-eigs-dom}
\lvert \lambda_1\rvert \geq \ldots \geq \lvert \lambda_p\rvert > \lvert \lambda_{p+1}\rvert \geq \ldots \geq \lvert \lambda_n\rvert
\end{equation}
where $\lambda_i$ are the eigenvalues of $A$.
The eigenspace associated to the $p$ dominant eigenvalues of $A$ is a $p$-dimensional attractor of the system.
(We do not provide any proof for this claim here but the reader will notice that \eqref{eq-eigs-dom} is a consequence of Theorem~\ref{thm-lyap} in Section~\ref{sec-algorithm}.)

\subsection{Path-dominant switching linear systems}

We extend $p$-dominance to study switching linear systems.
\emph{Switching linear systems} are linear time-dependent systems $x(t+1)=A(t)x(t)$ where the matrix $A(t)$ belongs to some \emph{finite} subset $\{A_1,\ldots,A_N\}\subseteq\RRb^{\ntn}$.
The set $\{A_1,\ldots,A_N\}$ is sometimes called the set of \emph{operating modes} of the system.
The system is constrained if the choice of $A(t)$ is constrained by the past choices of $A(s)$, $s<t$.

More formally, a switching linear system is a system of the form
\begin{equation}
\label{eq-sw-syst}
\PrintSwicth
\end{equation}
where,  $\sigma(t)\in\Sigma=\{1,\ldots,N\}$ and $A_{\sigma(t)}\in\RRb^\ntn$, for all $t$.
A function $x(\cdot):\ZZb\to\RRb^n$ is called a \emph{trajectory} of the system if there exists a signal $\sigma(\cdot):\ZZb\to\Sigma$, called a \emph{switching signal}, such that \eqref{eq-sw-syst} is satisfied for every $t\in\ZZb$.
\eqref{eq-sw-syst} is a \emph{constrained switching system} if the (bi-infinite) sequences $\{\ldots,\sigma(-1),\sigma(0),\sigma(1),\sigma(2),\ldots\}$ generated by the switching signal $\sigma(\cdot)$ belongs to some \emph{restricted language} $\calL\subsetneq\Sigma^\ZZb$.

A usual way to restrict the space of admissible switching signals is by the use of a \emph{finite-state automaton}.
Formally, a finite-state automaton $\ttAut$ is a triplet $(Q,\Sigma,\delta)$ where $Q$ is the (finite) set of states, $\Sigma=\{1,\ldots,N\}$ is the alphabet, and $\delta\subseteq Q\times\Sigma\times Q$ is the set of admissible transitions.
We will write $\MyTransition{q_1}{\sigma}{q_2}\in\delta$ if $(q_1,\sigma,q_2)\in\delta$.
A (bi-infinite) sequence (also called a \emph{word}) $\{\sigma(t)\}$, $\sigma(t)\in\Sigma$, $t\in\ZZb$, is \emph{admissible} for $\ttAut$ if there exists a (bi-infinite) sequence $\{q_t\}$, $q_t\in Q$, $t\in\ZZb$, such that $\MyTransition{q_t}{\sigma(t)}{q_{t+1}}\in\delta$ for every $t$.
An automaton $\ttAut$ is \emph{path-complete} for the language $\calL$ if any word that belongs to $\calL$ is admissible for $\ttAut$.
Examples of finite-state automatons are presented in Figure~\ref{fig-finite-state-automaton}.

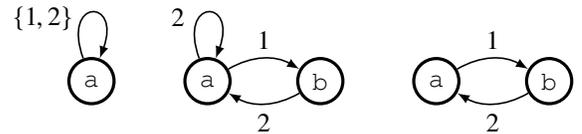
\begin{figure}[htbp]
\centering
\begin{tikzpicture}[baseline=(current bounding box.center)]
\tikzset{
  node distance=1.5cm,
  my arrow/.style={-latex,auto,semithick},
  my state/.style={draw=black,circle,inner sep=0pt,minimum size=6mm,very thick},
}

\node[my state] (A) {$\tta$};
\path (A) edge[my arrow,loop above,looseness=10,out=110,in=70] node[pos=0.15,above left= 2pt and 2pt,inner sep=0pt,outer sep=0pt] {$\{1,2\}$} (A);

\node[my state,right=9mm of A] (A) {$\tta$};
\node[my state,right of=A] (B) {$\ttb$};
\path (A) edge[my arrow,bend left] node[auto] {$1$} (B);
\path (B) edge[my arrow,bend left] node[auto] {$2$} (A);
\path (A) edge[my arrow,loop above,looseness=10,out=110,in=70] node[pos=0.15,above left] {$2$} (A);

\node[my state,right=9mm of B] (A) {$\tta$};
\node[my state,right of=A] (B) {$\ttb$};
\path (A) edge[my arrow,bend left] node[auto] {$1$} (B);
\path (B) edge[my arrow,bend left] node[auto] {$2$} (A);
\end{tikzpicture}
\vskip-2pt
\caption{Three automatons with $\Sigma=\{1,2\}$ and $Q=\{\tta\}$ (for the left-most automaton) or $Q=\{\tta,\ttb\}$ (for the central and right-most automatons).
The automaton on the right accepts every word with a strict alternation of $1$ and $2$.
The automaton in the middle accepts every word that contains no two consecutive $1$'s.
The automaton on the left accepts every word on the alphabet $\{1,2\}$.
Indeed, every word admissible for the right-most automaton is also admissible for the automatons in the middle and on the left.}
\label{fig-finite-state-automaton}
\end{figure}

\emph{Path $p$-dominance} extends $p$-dominance to switching systems by considering a set of quadratic cones $\calKP{q}$ and a contraction property adapted to constrained sequences of switches $\{\ldots,\sigma(-1),\sigma(0),\sigma(1),\sigma(2),\ldots\}$.

\begin{definition}\label{def-path-p-domi}
The switching system \eqref{eq-sw-syst} is \emph{$p$-dominant with respect to the automaton} $\ttAut=(Q,\Sigma,\delta)$ if there exists a set of quadratic $p$-cones $\{\calKP{q}\}$, $q\in Q$, such that
\begin{equation}
\label{eq-contra-path}
A_\sigma\big(\calKP{q_1}\saufzero\big) \subseteq \nint\,\calKP{q_2}
\end{equation}
for every transition $\MyTransition{q_1}{\sigma}{q_2}\in\delta$.

For any given language $\calL$, \eqref{eq-sw-syst} is \emph{path-complete $p$-dominant} if there exists an automaton $\ttAut$ such that $\ttAut$ is path-complete for $\calL$ and \eqref{eq-sw-syst} is $p$-dominant with respect to $\ttAut$.~\hfill$\lrcorner$
\end{definition}

Condition \eqref{eq-contra-path} expresses the property that, for every transition $\MyTransition{q_1}{\sigma}{q_2}\in\delta$, $\calKP{q_1}$ is \emph{contracted} into $\calKP{q_2}$ by $A_\sigma$.
Path-dominance admits an algebraic characterization based on LMIs:

\begin{proposition}\label{pro-LMI-path}
The switching system \eqref{eq-sw-syst} is $p$-dominant with respect to the automaton $\ttAut=(Q,\Sigma,\delta)$ if and only if there exist (i)~a set of symmetric matrices $\{P_q\}$, $P_q\in\RRb^\ntn$, $q\in Q$, with uniform inertia $(p,0,n-p)$, and (ii)~a set of \emph{rates} $\{\gamma_d\}$, $\gamma_d>0$, $d\in\delta$, and an $\varepsilon>0$ such that
\begin{equation}
\label{eq-LMI-path}
A_\sigma^\top P_{q_2}^{}A_\sigma^{} - \gamma_d^2\, P_{q_1}^{} \preceq -\varepsilon I
\end{equation}
for every transition $d=\MyTransition{q_1}{\sigma}{q_2}\in\delta$.~\hfill$\lrcorner$
\end{proposition}

\begin{proof}
Similar to the proof of Proposition~\ref{pro-LMI-S-proc}, adapting the argument to several cones.
\end{proof}

\begin{example}\label{exa-no-common}
Consider the system \eqref{eq-sw-syst} with $\Sigma=\{1,2\}$ and 
\[
A_1 = {\setMatrix\left[\begin{array}{cc} 2 & 0 \\ 0 & 4 \end{array}\right]}
\quad\text{and}\quad
A_2 = {\setMatrix\left[\begin{array}{cc} 1 & 0 \\ 0 & 1/8 \end{array}\right]} .
\]
Assume that the switching signal $\sigma(\cdot):\ZZb\to\Sigma=\{1,2\}$ is constrained to be a strict alternation of $1$ and $2$, represented by the right-most automaton in Figure~\ref{fig-finite-state-automaton}.
We show that the system is path $1$-dominant.

Consider the set of symmetric matrices 
\[
P_\tta = {\setMatrix\left[\begin{array}{cc} -1 & 0 \\ 0 & 8 \end{array}\right]}
\quad\text{and}\quad
P_\ttb = {\setMatrix\left[\begin{array}{cc} -1/2 & 0 \\ 0 & 1/4 \end{array}\right]} ,
\]
both with inertia $(1,0,1)$, and take uniform rates $\gamma_d=1$ for all $d\in\delta$.
\eqref{eq-LMI-path} is satisfied for every transition: in particular, for $d=\MyTransitionNum{\tta}{1}{\ttb}$, we get
\begin{align*}
A_1^\top P_\ttb^{}\,A_1^{} - P_\tta^{} = {\setMatrix\left[\begin{array}{cc} -2 & 0 \\ 0 & 4 \end{array}\right]} - {\setMatrix\left[\begin{array}{cc} -1 & 0 \\ 0 & 8 \end{array}\right]} = {\setMatrix\left[\begin{array}{cc} -1 & 0 \\ 0 & -4 \end{array}\right]} ,
\end{align*}
and for $d=\MyTransitionNum{\ttb}{2}{\tta}$, we get
\[
\renewcommand{\setMatrix}{\small\arraycolsep=0.3\arraycolsep}%
A_2^\top P_\tta^{}\,A_2^{} - P_\ttb^{} = {\setMatrix\left[\begin{array}{cc} -1 & 0 \\ 0 & 1/8 \end{array}\right]} - {\setMatrix\left[\begin{array}{cc} -1/2 & 0 \\ 0 & 1/4 \end{array}\right]} = {\setMatrix\left[\begin{array}{cc} -1/2 & 0 \\ 0 & -1/4 \end{array}\right]} .
\]

The cones $\calKP{\tta}$ and $\calKP{\ttb}$ are depicted in Figure~\ref{fig-path-complete-cones}.
We observe that the system requires $\calKP{\tta} \neq \calKP{\ttb}$.
\eqref{eq-contra-path} would not be feasible with the additional constraint $\calKP{\tta} = \calKP{\ttb}$.~\hfill$\lrcorner$
\end{example}

\begin{figure}[h!]
\centering
\begin{tikzpicture}
\node[draw=none,inner sep=0pt] at (0,0) {\includegraphics[height=4.2cm]{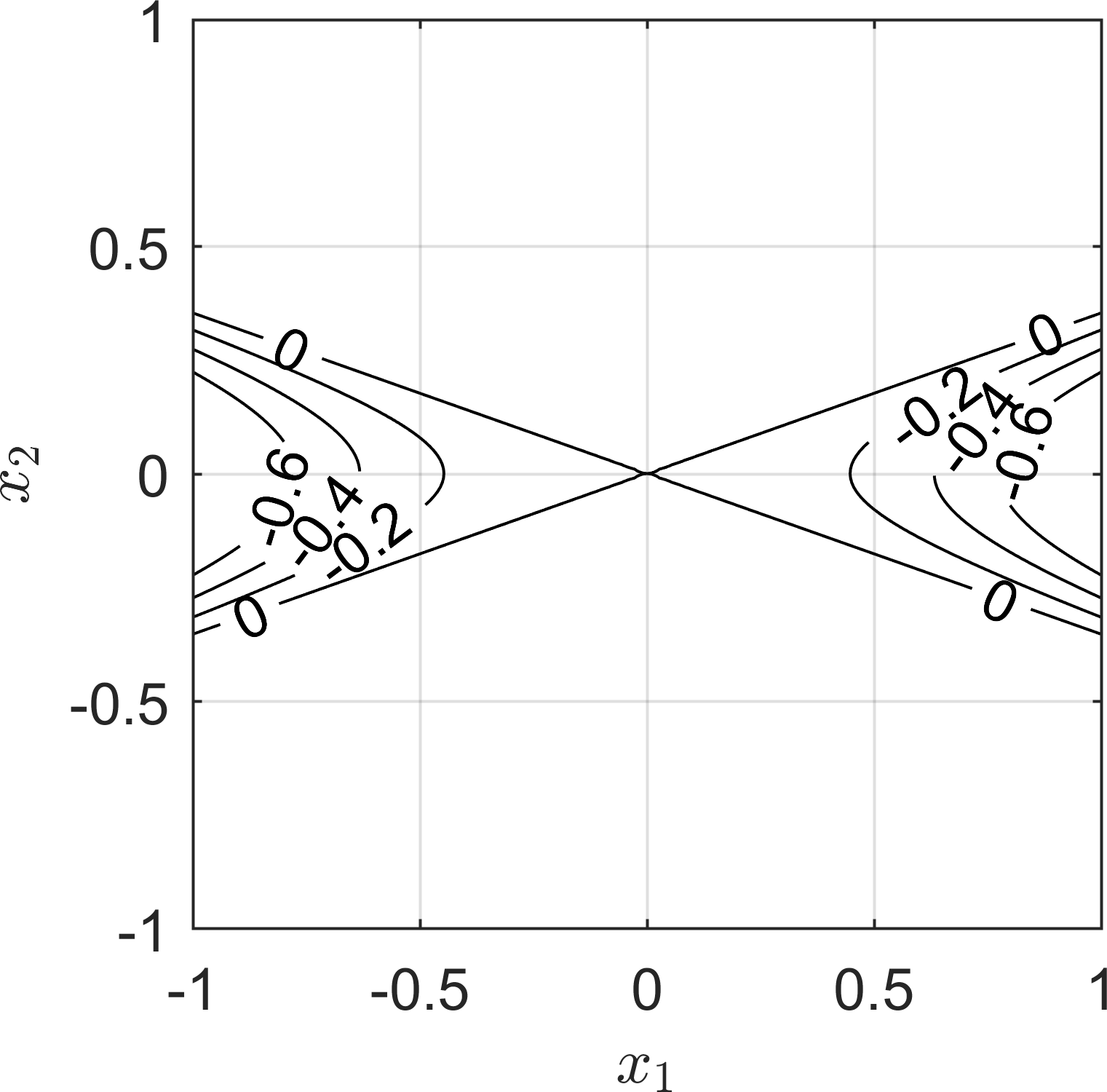}};
\node[draw=none,inner sep=0pt] at (4.3,0) {\includegraphics[height=4.2cm]{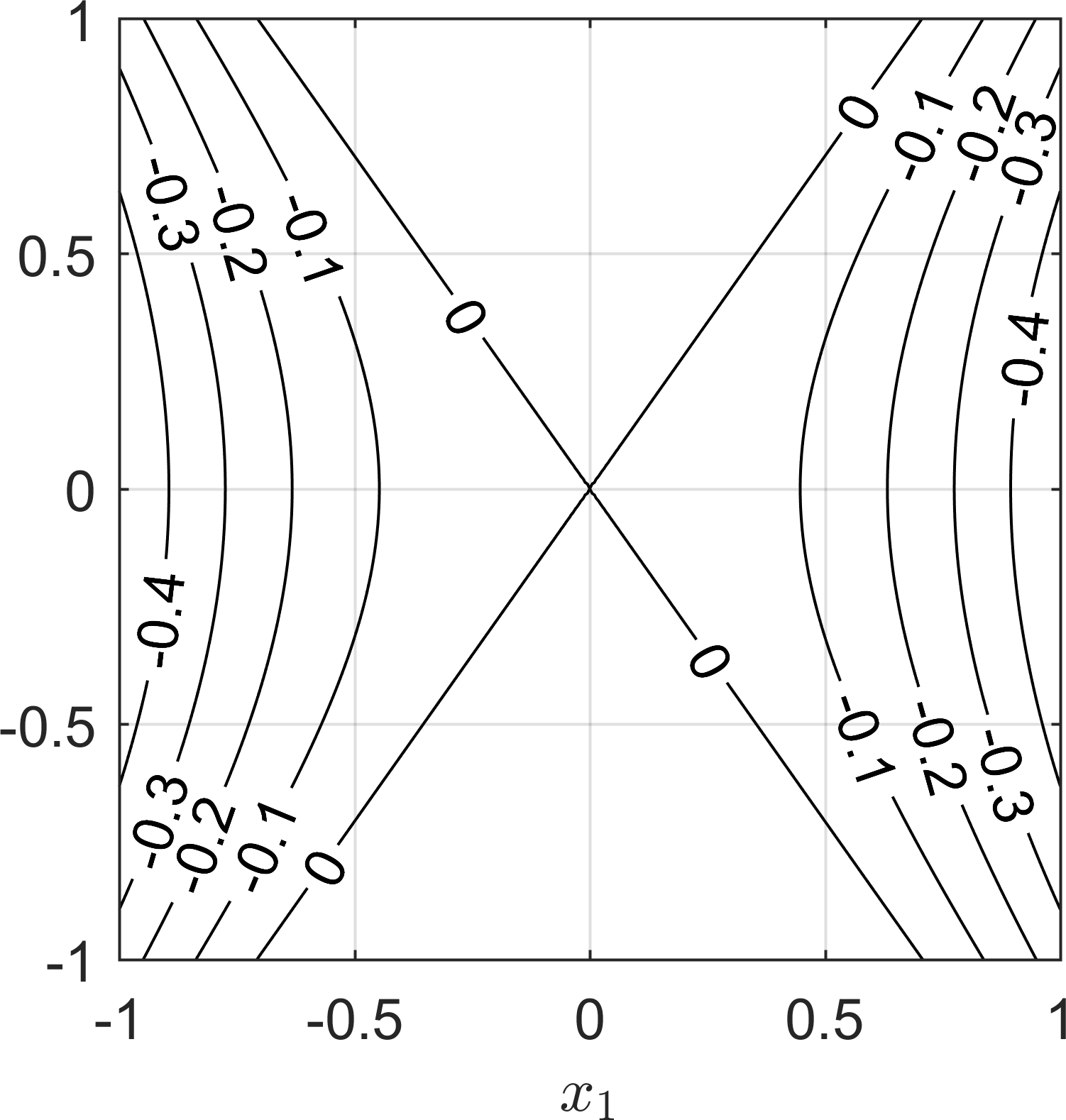}};
\node[fill=none] at (0.4,1.5) {$P_\tta$};
\node[fill=none] at (4.56,1.5) {$P_\ttb$};
\end{tikzpicture}
\vskip-5pt
\caption{Negative level curves of $x^\top P_\tta x$ and $x^\top P_\ttb x$.}
\label{fig-path-complete-cones}
\end{figure}

\section{THE ASYMPTOTIC BEHAVIOR OF\\PATH-DOMINANT SYSTEMS}
\label{sec-behavior}

In this section, we study the asymptotic behavior of path $p$-dominant switching systems \eqref{eq-sw-syst} with admissible language $\calL$.
Path dominance guarantees low-dimensional asymptotic behavior.
For each switching sequence $\sigma(\cdot)\in\calL$, there exists a collection of $p$-dimensional linear subspaces $H(t)\subseteq\RRb^n$, $t\in\ZZb$, that are invariant for the system dynamics, that is, each subspace $H(t)$ is mapped by the system into $H(t+1)$.
These subspaces are also globally attractive, as clarified in Theorem~\ref{thm-splitting-sw} below.

The following notation and definitions will simplify the exposition.
For a given $\sigma(\cdot)\in\calL$, and for $s,t\in\ZZb$, $s\leq t$, let $\sigma(s,t)=\{\sigma(s),\sigma(s+1),\ldots,\sigma(t)\}$ be a finite \emph{subsequence} of $\sigma(\cdot)$ and define 
\begin{equation}
A_{\sigma(s,t)}=A_{\sigma(t)}A_{\sigma(t-1)}\cdots A_{\sigma(s+1)}A_{\sigma(s)} .
\end{equation}
Note that if $x(\cdot)$ is a trajectory of the system with switching signal $\sigma(\cdot)$, then $x(t)=A_{\sigma(s,t-1)} x(s)$ for every $t>s$.
Define two \emph{collections of linear subspaces}
\begin{equation}
\begin{split}
\calH &= \{\ldots,H(-1),H(0),H(1),\ldots\}, \\
\calV &= \{\ldots,V(-1),V(0),V(1),\ldots\}.
\end{split}
\end{equation}
We say that $\calH$ and $\calV$ are \emph{paired collections} if $\RRb^n=H(t) \oplus V(t)$ for every $t\in\ZZb$.
A collection $\calH$ (resp.\ $\calV$) is said to be \emph{of dimension $p$} if  $H(t)$ (resp.\ $V(t)$) has dimension $p$ for every $t\in\ZZb$.
$\calH$ (resp.\ $\calV$) is said \emph{forward invariant} for $\sigma(\cdot)$ if $A_{\sigma(t)}H(t)\subseteq H(t+1)$ (resp.\ $A_{\sigma(t)}V(t)\subseteq V(t+1)$) for every $t\in\ZZb$.
Given two paired collections $\calH$ and $\calV$, we denote by $\Hhat(t):\RRb^n\to H(t)$ the projection on $H(t)$ along $V(t)$.
Similarly, $\Vhat(t):\RRb^n\to V(t)$ denotes the projection on $V(t)$ along $H(t)$.

\begin{theorem}[Dominated splitting]\label{thm-splitting-sw}
Given any language $\calL$, let \eqref{eq-sw-syst} be \emph{path $p$-dominant}.
Then for any signal $\sigma(\cdot) \in \calL$,
\begin{itemize}
\item
there exist unique, forward invariant, paired collections $\calH$ and $\calV$ of dimension $p$ and $n-p$, respectively;
\item
there exist $\rho<1$ and $C\geq 1$ such that, for every $s,t\in\ZZb$, $s\leq t$, and every $x\in\RRb^n$,
\begin{equation}
\label{eq-growth-relative}
\frac{\big\lvert\, \Vhat(t) A_{\sigma(s,t-1)} x\,\big\rvert}{\big\lvert\, \Hhat(t) A_{\sigma(s,t-1)} x\,\big\rvert} 
\leq C \rho^{t-s} \frac{\big\lvert\, \Vhat(s) x\,\big\rvert}{\big\lvert\, \Hhat(s) x\,\big\rvert} \,.
\end{equation} 
\end{itemize}
Note that $\calH$ and $\calV$ both depend on $\sigma(\cdot)$.~\hfill$\lrcorner$
\end{theorem}

The details of the proof of Theorem~\ref{thm-splitting-sw} are left to an extended version of this paper but the proof closely follows the arguments in \cite[Theorem~1]{bib-ForSep17DDTD} and \cite{bib-New04CFDH}.

The interpretation of Theorem~\ref{thm-splitting-sw} is that, for almost every trajectories $x(\cdot)$ of the system \eqref{eq-sw-syst}, the component $x_v(t) = \Vhat(t) x(t)$ becomes negligible compared to the component $x_h(t) = \Hhat(t)x(t)$.
$p$-dominance does not tell anything about the absolute growth of $x_v(t)$ and $x_h(t)$.
However, if $x(t)$ is uniformly bounded, then $x_v(\cdot)$ converges asymptotically to zero.
Furthermore, if $p$-dominance is established by \eqref{eq-LMI-path} with rates $\gamma_d\leq1$, this means that every trajectories $x(t)$ asymptotically converge to the invariant collection $H(t)$.

\begin{remark}
If $\sigma(\cdot)$ is periodic, we can show that $\calH$ necessarily contains only a finite number of different subspaces, which define a \emph{periodic attractor} for the system.
See also Example~\ref{exa-path-no-common} below.~\hfill$\lrcorner$
\end{remark}

\begin{remark}
For simplicity, our results rely on bi-infinite switching sequences, defined at every time $t$ from $-\infty$ to $\infty$.
This is classical in dynamical systems in order to describe the steady-state behavior.
But all the results can readily be interpreted for trajectories starting at time $t=0$.
For instance, take a $p$-dominant system \eqref{eq-sw-syst} with respect to the automaton $\ttAut=(Q,\Sigma,\delta)$, and let $\sigma(\cdot)$ be admissible for $\ttAut$.
Then any trajectory $x(\cdot)$ admits a decomposition into $x(t) = x_h(t) + x_v(t)$ where $x_h(t) \in \calK(P_{q_t})$, $x_v(t) \in V(t)$, and both trajectories $x_h(t+1) = A_{\sigma(t)} x_h(t)$, $x_v(t+1) = A_{\sigma(t)} x_v(t)$ satisfy $\lim_{t\to\infty} \frac{|x_v(t)|}{|x_h(t)|} = 0$.~\hfill$\lrcorner$
\end{remark}

The projective contraction property \eqref{eq-growth-relative} is illustrated by the normalized trajectories in Figure~\ref{fig-plane-contraction}, which shows the behavior of a path $2$-dominant system whose trajectories are radially scaled to the unit sphere.
The $2$-dominant behavior of the system is captured by the convergence to a two-dimensional plane.

Figure~\ref{fig-plane-contraction} illustrates the key difference between stability and dominance.
The trajectories of a stable switching system all converge to a unique equilibrium.
A $p$-dominant system allows for richer behaviors.
Its trajectories all converge to a $p$-dimensional hyperplane.

\begin{figure}[h!]
\centering
\includegraphics[width=\linewidth]{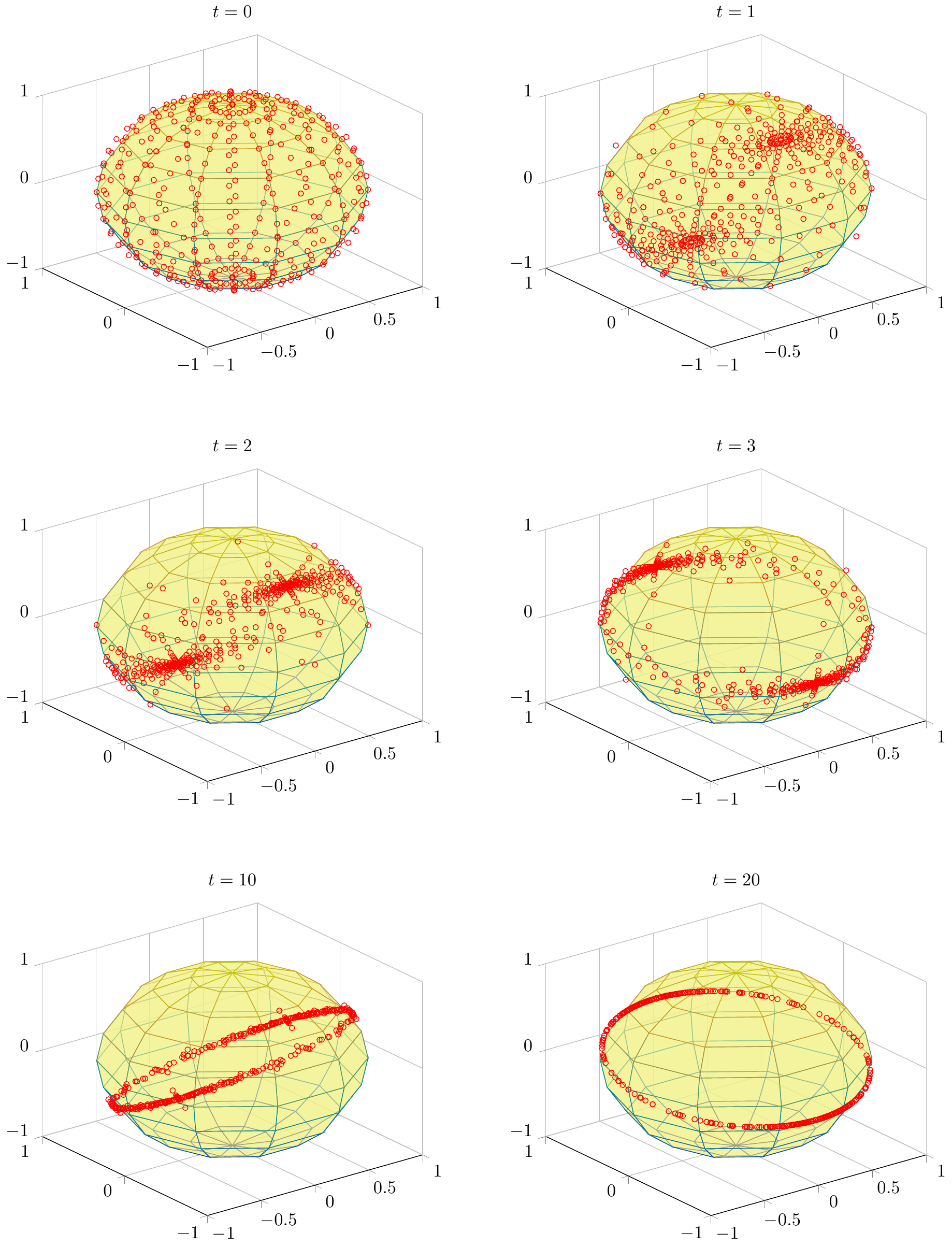}
\vskip-5pt
\caption{Trajectories of a $2$-dominant switching system from different initial conditions and for a fixed signal $\sigma(\cdot)$.
Each red dot represents the projection on the sphere of a trajectory $x(\cdot)$ at different times $t$.}
\label{fig-plane-contraction}
\end{figure}

\section{ALGORITHM}
\label{sec-algorithm}

It turns out that path-dominance can be tested algorithmically.
Consider \eqref{eq-sw-syst} and let $\ttAut=(Q,\Sigma,\delta)$ be a path-complete automaton for the language of \eqref{eq-sw-syst}.
By Proposition~\ref{pro-LMI-path}, given a set of rates $\{\gamma_d\}$, $\gamma_d>0$, $d\in\delta$, checking $p$-dominance of \eqref{eq-sw-syst} requires to solve the following problem:

\begin{problem}\label{pro-find-Pq}
Find symmetric matrices $\{P_q\}$, $q\in Q$, with inertia $(p,0,n-p)$ and an $\varepsilon>0$ such that
\begin{equation}
\label{eq-LMI-PROB}
A_\sigma^\top P_{q_2}^{}A_\sigma^{} - \gamma_d^2\, P_{q_1}^{} \preceq -\varepsilon I
\end{equation}
for every transition $d=\MyTransition{q_1}{\sigma}{q_2}\in\delta$.~\hfill$\lrcorner$
\end{problem}

\eqref{eq-LMI-PROB} is a semi-definite program (SDP) but the constraints on the inertia of $P_q$ cannot be expressed as a semi-definite program (non-convex optimization).
However, we show below that the inertia constraints can be dropped in our case (this is essentially due to the forward and backward completeness of our automaton as we will see below).
The following result is instrumental: (In what follows, $\nu(P_q)$ denotes the number of negative eigenvalues of $P_q$.)

\begin{proposition}\label{pro-order-morphisme}
Let $(\{P_q\},\varepsilon)$ be an admissible solution of \eqref{eq-LMI-PROB} with $\varepsilon>0$.
If there is a transition $\MyTransition{q_1}{\sigma}{q_2}\in\delta$ for some $\sigma\in\Sigma$, then $\nu(P_{q_1})\leq\nu(P_{q_2})$.~\hfill$\lrcorner$
\end{proposition}

\begin{proof}
There exists a $\nu(P_{q_1})$-dimensional subspace $\calU$ such that $x^\top P_{q_1}x\leq0$ for all $x\in\calU$.
From \eqref{eq-LMI-PROB}, we have $y^\top P_{q_2}y<0$ for every $y=A_\sigma x$, $x\in\calU\saufzero$.
This implies that $P_{q_2}$ is negative definite on $A_\sigma\calU$ and also that $\dime\,A_\sigma\calU=\nu(P_{q_1})$ because $A_\sigma x\neq0$ for every $x\in\calU\saufzero$.
We conclude that $P_{q_2}$ has at least $\nu(P_{q_1})$ negative eigenvalues (see, e.g., \cite[Theorem~4.2.6]{bib-HorJoh90}).
\end{proof}

Proposition~\ref{pro-order-morphisme} actually shows that there is an order-preserving map between the states $q\in Q$ (with the partial order induced by the graph of $\ttAut$) and the number $\nu(P_q)$ of negative eigenvalues of $P_q$.
This turns out to be very useful for the resolution of Problem~\ref{pro-find-Pq}, as illustrated with  Example~\ref{exa-path-graph} below.
(General conclusions can be deduced from the example.
For reasons of space, we do not formalize this result.)

\begin{example}\label{exa-path-graph}
Consider the automaton depicted in Figure~\ref{fig-automaton-graph}.
Let $\{\gamma_d\}$ be a given set of rates, and assume that $(\{P_q\},\varepsilon)$ is an admissible solution of \eqref{eq-LMI-PROB} with $\varepsilon>0$.
Then from Proposition~\ref{pro-order-morphisme}, we have the following inequalities: $\nu(P_\tta)\leq\nu(P_\ttb)$, $\nu(P_\ttb)\leq\nu(P_\ttc)$, $\nu(P_\ttc)\leq\nu(P_\tta)$, $\nu(P_\tta)\leq\nu(P_\ttc)$, $\nu(P_\ttb)\leq\nu(P_\ttd)$, and $\nu(P_\ttd)\leq\nu(P_\tte)$.
Hence, $P_\tta$, $P_\ttb$ and $P_\ttc$ have the same number of negative eigenvalues, let's say $p_1$; and we have $p_1\!\leq\! p_2 \!\leq\! p_3$ where $p_2\!=\!\nu(P_\ttd)$ and $p_3\!=\!\nu(P_\tte)$.~\hfill$\lrcorner$
\end{example}

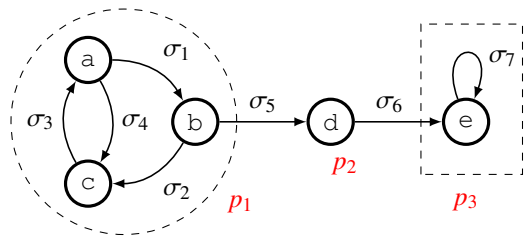
\begin{figure}[htbp]
\centering
\begin{tikzpicture}[baseline=(current bounding box.center)]
\tikzset{
  node distance=1.8cm,
  my arrow/.style={-latex,auto,semithick},
  my state/.style={draw=black,circle,inner sep=0pt,minimum size=6mm,very thick},
}

\pgfmathsetmacro{\rad}{0.95}
\node[my state] (B) at (0:\rad) {$\ttb$};
\node[my state] (A) at (120:\rad) {$\tta$};
\node[my state] (C) at (240:\rad) {$\ttc$};
\node[my state,right of=B] (D) {$\ttd$};
\node[my state,right of=D] (E) {$\tte$};
\path (A) edge[my arrow,bend left] node[auto] {$\sigma_1$} (B);
\path (B) edge[my arrow,bend left] node[auto] {$\sigma_2$} (C);
\path (C) edge[my arrow,bend left] node[auto] {$\sigma_3$} (A);
\path (A) edge[my arrow,bend left] node[auto] {$\sigma_4$} (C);
\path (B) edge[my arrow] node[pos=0.5] {$\sigma_5$} (D);
\path (D) edge[my arrow] node[pos=0.4] {$\sigma_6$} (E);
\path (E) edge[my arrow,loop above,looseness=10,out=110,in=70] node[pos=0.85,above right] {$\sigma_7$} (E);

\draw[dashed] (0,0) circle (1.5);
\coordinate (SW) at ($(E)+(-0.6,-0.7)$);
\coordinate (NE) at ($(E)+(0.75,1.3)$);
\draw[dashed] (SW) rectangle (NE);

\node[red] at (-35:1.9) {$p_1$};
\node[red,below right=3.5mm and -1mm of D.center] {$p_2$};
\node[red,below=5mm of E] {$p_3$};
\end{tikzpicture}
\vskip2pt
\caption{Automaton used in Example~\ref{exa-path-graph}.}
\label{fig-automaton-graph}
\end{figure}

Each set of nodes $q$ connected by a loop in $\ttAut$ must carry symmetric matrices $P_q$ with the \emph{same} $\nu(P_q)$.
Further results on the value of $\nu(P_q)$ can be obtained from the Lyapunov--Stein theorem below.
The key observation is that the value of $\nu(P_q)$ at each node $q\in Q$ is regulated by the selection of the rates $\gamma_d$, $d\in \delta$.

Looking at Example~\ref{exa-path-graph}, we wonder whether $\nu(P_q)$ can take different values,
and also how we can force $\nu(P_q)$ to be equal to $p$ for all $q\in Q$.
In the case of Example~\ref{exa-path-graph}, these questions are equivalent to: ``can $p_1,p_2,p_3$ have different values; and how can we ensure $p_1=p_2=p_3=p$?''
The answer relies on the following theorem \cite[Section~13.2]{bib-LanTis85TMWA}:

\begin{theorem}[Lyapunov--Stein]\label{thm-lyap}
Let $A\in\RRb^\ntn$.
Then there exists a symmetric matrix $P\in\RRb^\ntn$ satisfying
\[
A^\top PA - P \prec 0
\]
if and only if $A$ has no eigenvalues on the unit circle (in the complex plane).
Moreover, in this case, $P$ has inertia $(k,0,n-k)$, where $k$ is the number of eigenvalues of $A$ that are \emph{strictly outside} the unit circle.~\hfill$\lrcorner$
\end{theorem}

We use Theorem~\ref{thm-lyap} to continue the analysis of Example~\ref{exa-path-graph}.
In what follows, $\Delta(A)$ denotes the triplet $(i_{\mathrm{out}},i_{\mathrm{on}},i_{\mathrm{in}})$, where $i_{\mathrm{out}}$ (resp.\ $i_{\mathrm{on}}$ and $i_{\mathrm{in}}$) is the number of eigenvalues of $A\in\RRb^\ntn$ that are \textit{strictly outside} (resp.\ \textit{on} and \textit{strictly inside}) the unit circle.
Let $\gamma_{d_7}>0$ be the rate associated to transition $d_7=\MyTransition{\tte}{\sigma_7}{\tte}$ in the automaton of Figure~\ref{fig-automaton-graph}.
If $(\{P_q\},\varepsilon)$ is an admissible solution of \eqref{eq-LMI-PROB} with $\varepsilon>0$, then $A_{\sigma_7}^\top P_\tte^{} A_{\sigma_7}^{} - \gamma_{d_7}^2\,P_\tte^{}\prec0$.
Hence, by Theorem~\ref{thm-lyap}, $\In(P_\tte)=\Delta(A_{\sigma_7}/\gamma_{d_7})$.
So there is only one possible value for $p_3$, and this value can be made equal to $p$ by choosing $\gamma_{d_7}$ such $\Delta(A_{\sigma_7}/\gamma_{d_7})=(p,0,n-p)$.

If we look at the closed loop $\tta\,\raisebox{-2pt}[0pt][0pt]{$\stackrel{\sigma_1}{\to}$}\,\ttb\,\raisebox{-2pt}[0pt][0pt]{$\stackrel{\sigma_2}{\to}$}\,\ttc\,\raisebox{-2pt}[0pt][0pt]{$\stackrel{\sigma_3}{\to}$}\,\tta$, we get the same conclusion: let $\gamma_{d_1},\gamma_{d_2},\gamma_{d_3}$ be the rates associated to the transitions $d_1=\MyTransition{\tta}{\sigma_1}{\ttb}$, $d_2=\MyTransition{\ttb}{\sigma_2}{\ttc}$ and $d_3=\MyTransition{\ttc}{\sigma_3}{\tte}$, and define the matrix $B=(A_{\sigma_3}A_{\sigma_2}A_{\sigma_1})/(\gamma_{d_3}\gamma_{d_2}\gamma_{d_1})$.
Then, since $B^\top P_\tta\,B - P_\tta \prec0$, we get that the inertia $\In(P_\tta)=\Delta(B)$.
Hence, $p_1$ is uniquely determined by the rates $\gamma_{d_1},\gamma_{d_2},\gamma_{d_3}$, and we have to choose them in such a way that $\Delta(B)=(p,0,n-p)$.

A similar reasoning holds for every closed loop in the graph of $\ttAut$.
Next to this, note that no closed loop constrains directly the value of $p_2$.
However, if we choose the rates $\gamma_d$ so that $p_1=p_3=p$ then, by $p_1\leq p_2\leq p_3$, 
we necessarily have $p_2=p$.

In conclusion, we have seen on a small example that the inertia of the matrices $P_q$ is constrained by the structure of the automaton.
Indeed, given the specific configuration rates $\gamma_d$, $d \in \delta$, if \eqref{eq-LMI-PROB} admits a solution with $\varepsilon>0$, then the inertia constraint is automatically satisfied.
This assumes that every node of the automaton is inside a loop (like $\tta$, $\ttb$, $\ttc$ and $\tte$ in Figure~\ref{fig-automaton-graph}), or is on a path between two nodes that are themselves inside loops (like $\ttd$ in Figure~\ref{fig-automaton-graph}).
This assumption is automatically satisfied if the automaton is path-complete for a non-empty language $\calL\subseteq\Sigma^\ZZb$.

The example below illustrates the algorithm on a simple application in biology:

\begin{example}\label{exa-path-no-common}
Consider a bacterial culture involving two bacteria.
$x_A(t)$ is the population of type $A$ bacteria and $x_B(t)$ the population of type $B$ bacteria at time $t$.
Let $x(t)$ be the vector $[x_A(t),x_B(t)]^\top$.
Three modes of action on the culture may arise:
\begin{itemize}
    \item \underline{Mode~1}: The input of a gene Ge transforms $90\%$ of type $B$ bacteria into type $A$ bacteria.
    \item \underline{Mode~2}: The removal of a gene Ge transforms $90\%$ of type $A$ bacteria into type $B$ bacteria.
    \item \underline{Mode~3}: After the transformation in Mode~1, we use a specialized antibiotics that kills bacteria not carrying Ge.
    For illustration, we assume that the antibiotics kills directly $50\%$ of $B$ and that the perturbation on the environment produces a proportional disruption on type $A$ bacteria.
    This action can be performed as many times as we want, before performing Mode~2 again.
\end{itemize}
For simplicity, we consider normalized populations, such that $x_A(t)+x_B(t)=1$ at each $t$.
Hence, the overall dynamics can be represented as the switching system \eqref{eq-sw-syst}  with the matrices
\[
A_1 = {\setMatrix\left[\begin{array}{cc} 1 & 0.9 \\ 0 & 0.1 \end{array}\right]} ,\;
A_2 = {\setMatrix\left[\begin{array}{cc} 0.1 & 0 \\ 0.9 & 1 \end{array}\right]} ,\;
A_3 = {\setMatrix\left[\begin{array}{cc} 1 & -0.5 \\ 0 & 0.5 \end{array}\right]} ,
\]
representing each mode respectively.
The allowed transitions of the system are the ones described by the following automaton (Mode~3 can be applied only if Mode~1 or Mode~3 was applied just before):
\[
\begin{tikzpicture}[baseline={([yshift=-4mm]current bounding box.center)}]
\tikzset{
  node distance=1.5cm,
  my arrow/.style={-latex,auto,semithick},
  my state/.style={draw=black,circle,inner sep=0pt,minimum size=6mm,very thick},
}
\node[my state,right=14mm of A] (A) {$\tta$};
\node[my state,right of=A] (B) {$\ttb$};
\path (A) edge[my arrow,bend left] node[auto] {$1$} (B);
\path (B) edge[my arrow,bend left] node[auto] {$2$} (A);
\path (A) edge[my arrow,loop above,looseness=10,out=110,in=70] node[pos=0.15,above left] {$2$} (A);
\path (B) edge[my arrow,loop above,looseness=10,out=110,in=70] node[pos=0.85,above right] {$1$} (B);
\path (B) edge[my arrow,loop above,looseness=10,out=20,in=-20] node[pos=0.35,above right=-1.5pt and -1.5pt] {$3$} (B);
\end{tikzpicture} \!\!.
\]
The system is path $1$-dominant, as certified by the feasibility of the CVX program \cite{bib-GraBoy08CVXM} below, based on the rates
\[
\begin{array}{ll}
d_1=\MyTransitionNum{a}{2}{a}:\, \gamma_{d_1}=3/4 , \hspace*{2mm} & d_3=\MyTransitionNum{b}{2}{a}:\, \gamma_{d_3}=1/4 , \\
d_2=\MyTransitionNum{a}{1}{b}:\, \gamma_{d_2}=1/4 , \hspace*{2mm} & d_4=\MyTransitionNum{b}{1}{b}:\, \gamma_{d_4}=3/4 , \\
& d_5=\MyTransitionNum{b}{3}{b}:\, \gamma_{d_5}=3/4 .
\end{array}
\]
\begin{lstlisting}
A1 = [1, 0.9; 0, 0.1];
A2 = [0.1, 0; 0.9, 1];
A3 = [1, -0.5; 0, 0.5];

cvx_begin sdp
variable Pa(2, 2) symmetric;
variable Pb(2, 2) symmetric;

A2' * Pa * A2 - 0.75^2 * Pa <= -0.01 * eye(2);
A1' * Pb * A1 - 0.25^2 * Pa <= -0.01 * eye(2);
A2' * Pa * A2 - 0.25^2 * Pb <= -0.01 * eye(2);
A1' * Pb * A1 - 0.75^2 * Pb <= -0.01 * eye(2);
A3' * Pb * A3 - 0.75^2 * Pb <= -0.01 * eye(2);
cvx_end
\end{lstlisting}


The choice of the rates $\gamma_d$ are somehow arbitrary but constrained by the relation that $\gamma_{d_1}$ must lie between $\lvert\lambda_2(A_2)\rvert$ and $\lvert\lambda_1(A_2)\rvert$ where $\lambda_1(A_2)$, $\lambda_2(A_2)$ are the eigenvalues of $A_2$ ordered with decreasing magnitude (see discussion above this example).
Similarly, $\lvert\lambda_2(A_1)\rvert<\gamma_{d_4}<\lvert\lambda_1(A_1)\rvert$, and $\lvert\lambda_2(A_3)\rvert<\gamma_{d_5}<\lvert\lambda_1(A_3)\rvert$.
This reduces considerably the search space of the $\gamma_d$.

$1$-dominance of the system asserts the existence of a family of $1$-dimensional subspaces that will attract all the trajectories.
If the switching sequence is periodic, e.g., $\sigma(\cdot)=\{\ldots2,1,3,2,1,3,\ldots\}$, then this family of attractors is finite.
For normalized trajectories, this implies the existence of a globally attractive periodic steady state solution, as shown in Figure~\ref{fig-example-attractors} (right).
Likewise, for constant switching, e.g., $\sigma(\cdot)=\{\ldots2,2,\ldots\}$, normalized trajectories necessarily converge to a fixed point, as shown in Figure~\ref{fig-example-attractors} (left).~\hfill$\lrcorner$
\end{example}

\begin{figure}[h!]
\centering
\includegraphics[width=\linewidth]{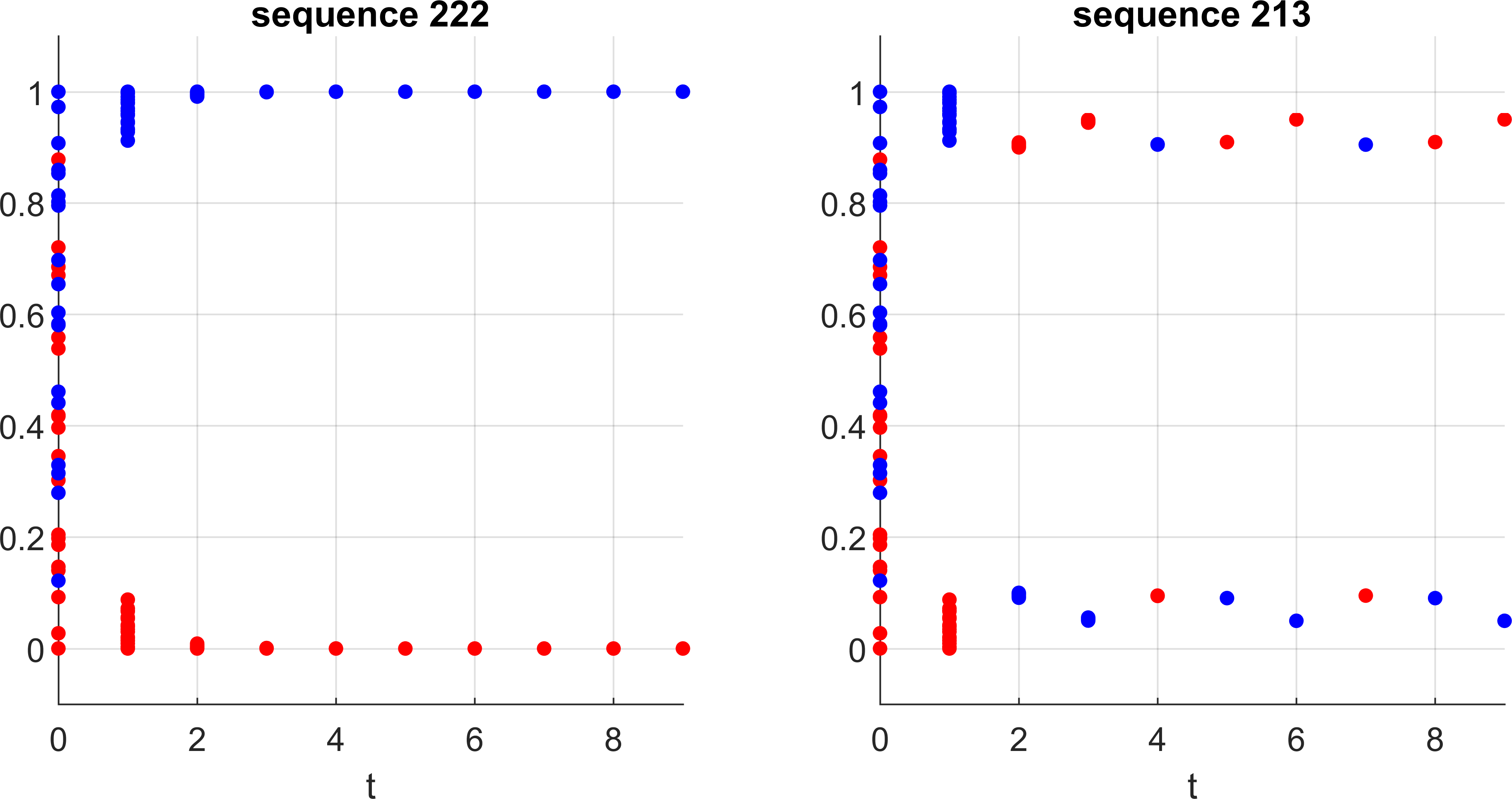}
\vskip-5pt
\caption{Trajectories in Example \ref{exa-path-no-common} from different initial conditions, with constant signal (left), and periodic signal (right).
Type $A$ and $B$ bacterial are in red and blue, respectively.
All trajectories quickly converge to the same orbit, as predicted by path $1$-dominance.}
\label{fig-example-attractors}
\end{figure}

\section{CONCLUSIONS}
\label{sec-application-discussion}

In this work, we have introduced the concept of \emph{path-complete $p$-dominance} for switching linear systems, which generalizes the notion of dominant/slow modes for LTI systems.
Our approach is geometric, as it relies on the contraction property of a finite set of cones.
We have shown that path-dominant switching systems inherit some of the nice properties of LTI systems with dominant/slow modes, namely that their asymptotic behavior is low-dimensional.

The algorithmic decidability of positivity/dominance for switching linear systems and nonlinear systems has already been addressed in \cite{bib-ForJun17PCPS} and \cite{bib-ForSep17DDTD}.
In comparison to path positivity \cite{bib-ForJun17PCPS} and dominance for nonlinear systems \cite{bib-ForSep17DDTD}, the main novelty of our algorithm is to make use of \emph{several} cones instead of searching for a single cone uniformly contracted by the system.
This allows us to capture a larger class of systems presenting a $p$-dominant behavior, and is particularly well suited for the analysis of switching linear systems with \emph{constrained} language.

We developed dominance as a direct extension of stability theory to capture richer asymptotic behaviors.
The use of path-complete \emph{Lyapunov functions} has become standard in the stability analysis of (constrained) switching linear systems \cite{bib-JunAhm17CLIS}, \cite{bib-AngAth17PCGC}.
Our algorithm draws upon these techniques but replaces contracting ellipsoids (or more generally, contracting norms) with contracting quadratic cones.
With our approach, we can study convergence to $p$-dimensional (time-varying) subspaces through an algorithm based on linear matrix inequalities.

An important restriction throughout this paper is to analyze dominance by the use of \emph{quadratic} cones.
This limitation is the price to pay for tractability.
Standard LMI solvers can then be used to construct the set of cones.
Note that, in the case of path-complete Lyapunov functions, it is well known that algebraic liftings exist, which allow to alleviate this restriction \cite{bib-PhiEss16SDTS}.
In the future, we plan to investigate this direction for path-dominance.

Due to space limitation, Section~\ref{sec-algorithm} provides a brief presentation of the algorithm and a short discussion on the choice of the rates $\gamma_d$.
These aspects will be discussed in future works.
The paper provides a few examples of path-dominant systems but a good effort is required to show the potential of the approach in applications.
We believe that our technique is well suited to tackle crucial problems in modern control.
For example, path-complete techniques allow to naturally model distributed and networked systems.

\vspace{2mm}

\noindent\textbf{Acknowledgement.}
The authors would like to thank R.~Sepulchre for insightful discussions on the paper.



\end{document}